\documentclass[a4paper,11pt,reqno]{amsart}

\usepackage[utf8]{inputenc}
\usepackage[T1]{fontenc}

\usepackage{amssymb,mathptmx}
\usepackage{microtype,cite}

\usepackage[colorlinks,linkcolor=blue,citecolor=blue,urlcolor=blue]{hyperref}
\theoremstyle{plain}
\newtheorem{thm}{Theorem}

\newtheorem{proposition}[thm]{Proposition}
\newtheorem{corollary}[thm]{Corollary}

\theoremstyle{definition}
\newtheorem{definition}[thm]{Definition}

\theoremstyle{remark}
\newtheorem{remark}[thm]{Remark}
\newtheorem{example}[thm]{Example}

\newcommand{\FF}{\mathcal{F}}

\newcommand{\NN}{\mathbb{N}}

\newcommand{\QQ}{\mathbb{Q}}
\newcommand{\RR}{\mathbb{R}}

\DeclareMathOperator{\EE}{\mathbb{E}}

\newcommand{\deq}{\mathrel{\mathop:}=}

\newcommand*{\one}{\text{\usefont{U}{bbold}{m}{n}1}}

\newcommand{\abs}[1]{\left|{#1}\right|}
\newcommand{\floor}[1]{\lfloor{#1}\rfloor}

\newcommand{\iid}{i.{\kern1pt}i.{\kern1pt}d.}
\newcommand{\df}{\ensuremath{d\mkern-1mu f}}

\renewcommand{\,}{\ifmmode\mkern2mu\else\thinspace\fi}
\renewcommand{\mid}{\ensuremath{\mkern2mu|\mkern2mu}}

\begin{document}

\title[On the normality of the concatenated Fibonacci constant]{On the normality of the concatenated \\ Fibonacci constant}

\author[J. R. G. Mendon\c{c}a]{Jos\'{e} Ricardo G. Mendon\c{c}a}
\address{Universidade de S\~{a}o Paulo, S\~{a}o Paulo, SP, Brasil}
\email{jricardo@usp.br}


\begin{abstract}
We study the concatenated Fibonacci constant $\mathcal{F} \deq 0.F_{1}F_{2}F_{3}\cdots = 0.11235813\cdots$, obtained by concatenating the Fibonacci numbers in the fractional part, and ask whether it is normal. We show that several classical sufficient conditions for normality by concatenation do not apply to the Fibonacci sequence because of its exponential growth, while a criterion of Pollack and Vandehey implies that the normality of $\mathcal{F}$ in base $10$ would follow if almost all Fibonacci numbers were $(\varepsilon,k)$-normal in base $10$. The Benford bias of leading digits and the Pisano periodicity of trailing digits are shown to contribute asymptotically negligible fractions of the total digits, isolating the distribution of the deep digits of large Fibonacci numbers as the remaining obstruction. Large-scale numerical experiments on the first $500{,}000$ Fibonacci numbers in bases $10$ and $2$ indicate that global single-digit counts and $k$-block statistics for $k = 2, 3, 4$ are compatible with \iid-like fluctuations at the scales tested, and that a positional decomposition concentrates the visible structured deviation at the boundaries between consecutive Fibonacci numbers, while pooled interior blocks remain close to uniform. Our computations suggest that any obstruction to normality lies in the asymptotic behavior of the deep digits of $F_{n}$.
\end{abstract}



\maketitle


\section{Introduction}
\label{sec:intro}

An irrational number $\alpha \in \RR \setminus \QQ$ is normal in the integer base $b\geq 2$ if, for every ${k \in \NN}$, every finite string $s \in \{0,1,\dots,b-1\}^{k}$ occurs in the base-$b$ expansion of $\alpha$ with asymptotic frequency $b^{-k}$. If this condition holds for $k = 1$, the number is said to be simply normal in base $b$; a number that is normal in every base is called absolutely normal. Borel proved in 1909 that almost every real number is absolutely normal \cite{borel}, yet proving that a given number is normal remains notoriously difficult. Well-known constructive examples of decimal normal numbers are Champernowne's constant $0.123456\cdots$ \cite{champernowne}, the Besicovitch number $0.149162536\cdots$ \cite{besicovitch}, and the Copeland--Erd\H{o}s constant $0.23571113\cdots$ \cite{copeland-erdos}, all obtained by concatenating sequences of integers whose counting functions grow sufficiently fast. While similar artificial numbers can be constructed in other bases \cite{stoneham,levin,bailey-crandall}, numbers like $\pi$ and Ap\'ery's constant $\zeta(3)$ have not yet been proved normal in any base.

In this paper we study the concatenated Fibonacci constant, henceforth referred to simply as the Fibonacci constant,
\begin{equation}
\FF = 0.F_{1}F_{2}F_{3}\cdots = 0.11235813213455\cdots,
\end{equation}
where $F_{1} = 1$, $F_{2} = 1$, and $F_{n} = F_{n-1}+F_{n-2}$ for $n \geq 3$, and ask whether $\FF$ is normal in base $10$. Among exponentially-growing sequences, the powers of two $2, 4, 8, 16, \ldots$ give a simpler and better-known concatenation problem---the normality of $0.2\,4\,8\,16\,32\cdots$, and more generally of the base-$b$ concatenation of $a^{n}$ for a fixed integer $a \geq 2$, is itself open and beyond the classical criteria of Section~\ref{sec:classical}. We nonetheless take the Fibonacci sequence as our primary test case: it is the linear recurrence with the smallest nontrivial growth rate $\phi = \frac{1}{2}(1+\sqrt{5})$, its arithmetic is unusually rich (through Pisano periods, divisibility properties, and Carmichael's theorem on primitive prime divisors), and it is familiar enough that results about it can be stated and calibrated without specialized preliminaries. At the same time, its exponential growth places $\FF$ just beyond the reach of the classical Davenport--Erd\H{o}s machinery, which requires the $n$-th term of the sequence to have at most $O(n^{1-\delta})$ digits for some $\delta > 0$. This makes the Fibonacci constant a natural test case beyond the polynomial-growth regime covered by the classical theory.

This study has two complementary goals. On the theoretical side, we identify a precise sufficient route to normality and show that the most visible sources of digit bias, the Benford distribution of leading digits and the Pisano periodicity of trailing digits, are asymptotically negligible. This isolates the deep digits of large Fibonacci numbers as the remaining obstruction. On the computational side, we test this structural picture by means of relatively large-scale numerical experiments in bases $10$ and $2$. The resulting data show that the observable deviations are concentrated at the boundaries between consecutive Fibonacci numbers and provide no evidence of a persistent bias in the pooled interior digits.

This paper is organized as follows. In Section~\ref{sec:classical}, we review the classical sufficient conditions for normality by concatenation and show how each fails for the Fibonacci constant. In Section~\ref{sec:reduction}, we apply a criterion of Pollack and Vandehey to identify the $(\varepsilon,k)$-normality of almost all Fibonacci numbers as a sufficient condition for the normality of $\FF$, and explain that this condition is not an equivalence. Section~\ref{sec:structural} analyzes the structural sources of digit bias in $\FF$ and shows that both the Benford distribution of leading digits and the Pisano periodicity of trailing digits contribute an asymptotically negligible fraction of the total digits, while the distribution of the deep digits of large Fibonacci numbers (the obstacle to applying the aforementioned sufficient route) is argued to lie beyond current equidistribution techniques. We also discuss in this section the complementary row-versus-column perspective arising from recent work of Benfield and Manes and explain why the argument does not settle the concatenation question. Sections~\ref{sec:experiments} and~\ref{sec:base2} report numerical experiments in bases $10$ and $2$, respectively, on the first $500{,}000$ Fibonacci numbers, together with a positional decomposition separating interior from boundary blocks. The statistical tests presented are diagnostic rather than formal; see Remark~\ref{rem:stats}. Finally, Section~\ref{sec:concluding} contains concluding remarks and directions for further work.


\section{Classical sufficient conditions for normality by concatenation}
\label{sec:classical}

The following proposition establishes the counting function for the Fibonacci sequence, a simple result that is otherwise difficult to find in the standard literature.

\begin{proposition}
\label{prop:count}
The number of Fibonacci numbers up to $N \in \NN$, counted with multiplicity---that is, the number of indices $n \geq 1$ with $F_{n} \leq N$, so that the equal terms $F_{1} = F_{2} = 1$ are counted twice---is given by
\begin{equation}
\label{eq:phib}
\Phi_{F}(N) = \bigl\lfloor \log[\sqrt{5}(N+\tfrac{1}{2})]/\log{\phi} \bigr\rfloor,
\end{equation}
where $\phi = \frac{1}{2}(1+\sqrt{5})$ is the golden ratio.
\end{proposition}
\begin{proof}
The Binet formula gives $F_{n} = \lfloor \phi^{n}/\sqrt{5} + \tfrac{1}{2} \rfloor$ for $n \geq 1$, and since $\lfloor x \rfloor \leq N$ if and only if $x < N+1$,
\begin{equation}
F_{n} \leq N \iff n < \log_{\phi}[\sqrt{5}\,(N+\tfrac{1}{2})].
\end{equation}
The right-hand side above cannot be an integer. Indeed, if $\sqrt{5}\,(N+\tfrac{1}{2}) = \phi^{m}$ for some integer $m \geq 1$, then the identity $\phi^{m} = F_{m}\phi + F_{m-1}$ gives $F_{m} + 2F_{m-1} + F_{m}\sqrt{5} = (2N+1)\sqrt{5}$, which is impossible for $m \geq 1$ since $F_{m} \geq 1$ and $F_{m-1} \geq 0$. Therefore, $\Phi_{F}(N) = \sum_{n \geq 1} \one[F_{n} \leq N]$ is exactly the number of integers $n \geq 1$ satisfying ${n <\log_{\phi}[\sqrt{5}\,(N+\tfrac{1}{2})]}$, and we get~\eqref{eq:phib}.
\end{proof}
\noindent This simple result shows that the Fibonacci sequence is exponentially sparse ($\Phi_{F}(N) \sim \log N$) among the positive integers.

\subsection{The Copeland--Erd\H{o}s barrier}
\label{subsec:copeland}

The classical theorem of Copeland and Erd\H{o}s \cite{copeland-erdos} asserts that the concatenation of an increasing sequence $a_{1} < a_{2} < \cdots$ is normal in base $b$ provided the counting function $\#\{a_{i} \leq N\}$ exceeds $N^{\theta}$ for every $\theta < 1$ and all sufficiently large $N$. Since $\log{N}/N^{\theta} \to 0$ for any $\theta > 0$, Proposition~\ref{prop:count} shows that the Fibonacci sequence fails this condition by a large margin.

The converse of the Copeland--Erd\H{o}s theorem is clearly false: one can regroup the digits of the Copeland--Erd\H{o}s constant $0.23571113\cdots$ into blocks $a_{1} = 2$, $a_{2} = 35$, $a_{3} = 711$, \ldots, where $a_{n}$ has exactly $n$ digits. The resulting sequence has counting function $\Phi(N) \sim \log_{10}{N}$, yet its concatenation produces a normal number. The Copeland--Erd\H{o}s theorem can therefore neither prove nor disprove the normality of~$\FF$.


\subsection{The Davenport--Erd\H{o}s approach and extensions}
\label{subsec:davenport}

A natural attempt to circumvent the sparsity barrier is to appeal to the stronger theorem of Davenport and Erd\H{o}s \cite{davenport-erdos}, which shows that $0.p(1)p(2)p(3)\cdots$ is normal in base $10$ for any polynomial $p$ taking positive integer values at positive integer arguments. This was subsequently generalized by Nakai and Shiokawa \cite{nakai-shiokawa-1992,nakai-shiokawa-1990} to functions of the form $f(x) = \alpha_{n}x^{\beta_{n}} + \cdots + \alpha_{1}x^{\beta_{1}}$ with $\beta_{n} > \cdots > \beta_{1} \geq 0$ and $f(x) > 0$ for $x > 0$.

These results are proved using Weyl's estimates for exponential sums involving polynomials or generalized polynomials. One might hope to bring the Fibonacci sequence into this framework by finding a polynomial $p$ with $p(k) = F_{k}$ for all $k \geq 1$. However, no such polynomial exists: any polynomial of degree $d$ satisfies $p(k) = O(k^{d})$, while $F_{k} = \Theta(\phi^{k})$ grows exponentially. Lagrange interpolation produces a polynomial of degree $n-1$ through the points $(1,F_{1}), \ldots, (n,F_{n})$, but the Davenport--Erd\H{o}s theorem requires a {fixed} polynomial, and its proof mechanism breaks down completely when the degree grows with the number of terms.

The Nakai--Shiokawa generalization extends the class of admissible functions considerably, but all functions in their framework have at most generalized-polynomial growth. Exponential functions like $F_{n} \sim \phi^{n}/\sqrt{5}$ remain out of reach.

More recent work by Clanin and Rayman \cite{clanin-rayman} studies the Davenport--Erd\H{o}s and Nakai--Shiokawa theorems through the lens of finite-state dimension, according to which a sequence is normal if and only if its finite-state dimension equals $1$. Their results show that rational linear polynomials preserve finite-state dimension of Copeland--Erd\H{o}s sequences, while polynomials of degree $\geq 2$ can change it. The results are elegant, but cannot resolve the normality of $\FF$, since the Fibonacci sequence is not obtained by applying any polynomial to a set of integers.


\subsection{The $\varphi$--$\sigma$--$\lambda$ approach}
\label{subsec:phi-sig-lbd}

Pollack and Vandehey showed that if $f$ is any function formed by composing Euler's totient function $\varphi$, the sum-of-divisors function $\sigma$, or Carmichael's lambda-function $\lambda$, then the number $0.f(1)f(2)f(3)\dots$ obtained by concatenating the base $b$ digits of successive $f$-values is $b$-normal \cite{phi-sigma-lambda}. 

That the approach cannot help can be seen from a simple growth argument. Since $\varphi(n) \leq n$, $\lambda(n) \leq n$, and $\sigma(n) = O(n\log\log{n})$, any finite composition $f = f_{1} \circ f_{2} \circ \cdots \circ f_{k}$ of them grows at most like $n$ times a polylogarithmic factor, so that $\log f(n) = O(\log n)$, which is precisely the weakly-polynomial-growth hypothesis that underlies the Pollack--Vandehey approach \cite[condition~(1.3) and Theorem~1.1]{phi-sigma-lambda}. Since $\log F_{n} \sim n\log{\phi}$, the Fibonacci numbers violate this bound by an exponential margin. Hence $n \mapsto F_{n}$ lies entirely outside the class of functions to which \cite{phi-sigma-lambda} applies, and no composition of $\varphi$, $\sigma$, $\lambda$ can agree with it even eventually. 

The obstruction is in fact already visible at small values: for example, $F_{5} = 5$ is not in the range of $\varphi$, $\sigma$, or $\lambda$ (no odd number $> 1$ is a totient or a Carmichael value, and $5$ is not a sum of divisors), so the Fibonacci sequence cannot be reproduced as $f(1), f(2), \ldots$ for any composition $f$ of these functions, independent of the growth obstruction above.


\section{A sufficient condition via digit normality of Fibonacci numbers}
\label{sec:reduction}

In this section we state a sufficient condition for the normality of $\FF$ in terms of the digit statistics of the Fibonacci numbers.

\begin{definition}
A natural number $n$ is $(\varepsilon,k)$-normal in base $b$ if
\begin{equation}
\abs{\frac{N(w,\sigma_{b}(n))}{\abs{\sigma_{b}(n)}}-b^{-k}} \leq \varepsilon
\end{equation}
for every string $w \in \{0, 1, \dots, b-1\}^{k}$, where $\sigma_{b}(n)$ denotes the base-$b$ representation of $n$ and $N(w,s)$ counts the number of sliding-block occurrences of $w$ in the string $s$.
\end{definition}

The notion of $(\varepsilon,k)$-normality goes back to Besicovitch \cite{besicovitch}, and the density-bound that underlies the criterion is due to Copeland and Erd\H{o}s \cite{copeland-erdos}. The following reformulation as a three-part sufficient condition for concatenations of arithmetic-function values is due to Pollack and Vandehey \cite{pollack-vandehey,phi-sigma-lambda}, with closely related variants appearing in \cite{normality-criterion,clanin-rayman}.

\begin{thm}[Normality criterion for concatenations]
\label{thm:criterion}
Let $a_{1}, a_{2}, a_{3}, \ldots$ be a sequence of positive integers. Suppose that as $m \to \infty$ the following conditions hold:
\begin{enumerate}
\item[($i$)] The digit lengths grow on average: $m = o\Bigl(\sum\limits_{n=1}^{m}\abs{\sigma_{b}(a_{n})}\Bigr)$.
\item[($ii$)] No single length dominates: $m \cdot\! \max\limits_{1 \leq n \leq m}\abs{\sigma_{b}(a_{n})} = O\Bigl(\sum\limits_{n=1}^{m}\abs{\sigma_{b}(a_{n})}\Bigr)$.
\item[($iii$)] For every $\varepsilon > 0$ and $k \in \NN$, the number of $n \leq m$ for which $a_{n}$ is not $(\varepsilon,k)$-normal is $o(m)$.
\end{enumerate}
Then $0.a_{1}a_{2}a_{3}\cdots$ is normal in base $b$.
\end{thm}

We prove that the Fibonacci sequence verifies conditions (i) and (ii) above.

\begin{proposition}
\label{prop:conditions}
The Fibonacci sequence $F_{1}, F_{2}, F_{3}, \ldots$ satisfies conditions~(i) and~(ii) of Theorem~\ref{thm:criterion} in base $b=10$.
\end{proposition}

\begin{proof}
Since $F_{n} \sim \phi^{n}/\sqrt{5}$, the number of base-$10$ digits of $F_{n}$ is
\begin{equation}
\abs{\sigma_{10}(F_{n})} = \floor{n \log_{10}{\phi} - \log_{10}{\sqrt{5}}} + 1 \sim n\log_{10}{\phi}
\end{equation}
for large $n$, where $\log_{10}{\phi} \approx 0.20898$.

\noindent Condition~(i): $\sum\limits_{n=1}^{m}\abs{\sigma_{10}(F_{n})} \sim \frac{1}{2}m^{2}\log_{10}{\phi}$, so $m = o(\frac{1}{2}m^{2}\log_{10}{\phi})$.

\noindent Condition~(ii): $m \cdot\! \max\limits_{1 \leq n \leq m}\abs{\sigma_{10}(F_{n})} = m\abs{\sigma_{10}(F_{m})} \sim m^{2}\log_{10}{\phi} = O(\frac{1}{2}m^{2}\log_{10}{\phi})$.
\end{proof}

By Theorem~\ref{thm:criterion}, the normality of $\FF$ is therefore implied by condition~(iii), i.e., that almost all Fibonacci numbers are $(\varepsilon,k)$-normal. We state this explicitly.

\begin{corollary}
\label{cor:reduction}
The Fibonacci constant $\FF$ is normal in base $10$ if, for every $\varepsilon > 0$ and $k \in \NN$,
\begin{equation}
\label{eq:condition}
\#\{n \leq m \colon F_{n} \text{ is not } (\varepsilon,k)\text{-normal in base } 10\} = o(m) \text{ as } m \to \infty.
\end{equation}
\end{corollary}

Condition~\eqref{eq:condition} is a statement about the digit distribution within individual Fibonacci numbers. Since $F_{n} \sim \phi^{n}/\sqrt{5}$, it is closely related to the question of whether the base-$10$ digits of $\phi^{n}$ are equidistributed for most $n$---though not identical to it, since condition~\eqref{eq:condition} concerns block statistics inside the finite integer strings $F_{n}$ and is sensitive to the rounding $F_{n} = \floor{\phi^{n}/\sqrt{5} + \tfrac{1}{2}}$ and to finite-length effects. It belongs to the same family as the open problem of whether the base-$b$ digits of $2^{n}$, or more generally $\alpha^{n}$ for fixed algebraic $\alpha > 1$, are asymptotically equidistributed for most $n$ \cite{bugeaud,kuipers-niederreiter}.


\section{Structural analysis of the digit distribution}
\label{sec:structural}

We show that two possible sources of digit bias in $\FF$, the Benford distribution of leading digits and the Pisano periodicity of trailing digits, are both negligible.

Before turning to the specific bases used in our numerical experiments, however, it is worth noting that the two structural mechanisms studied in this section have a general origin valid in every integer base $b \ge 2$. On the one hand, the leading digits of $F_n$ are governed by the fractional parts of $n\log_{b}{\phi} - \log_{b}{\sqrt{5}}$ and hence, by Weyl's equidistribution theorem, exhibit Benford-type behavior. On the other hand, the trailing $k$ digits are determined by $F_{n} \bmod b^{k}$ and are therefore controlled by the Pisano periodicity modulo powers of the base. Thus the boundaries between consecutive Fibonacci numbers carry deterministic structure in every base, although the precise form and strength of the resulting bias depend on the base and on the digit position.

\subsection{The leading digits}
Since $\log_{10}{\phi}$ is irrational, Weyl's equidistribution theorem guarantees that the fractional parts $\{n\log_{10}{\phi} + c\}$ (where $c = -\frac{1}{2}\log_{10}5$) are equidistributed modulo $1$. This implies that the leading digit of $F_{n}$ follows Benford's law, according to which digit $d$ appears as the first digit with frequency $\log_{10}(1+1/d)$; digit $1$, for instance, should appear $\sim 30\%$ of the time.

However, each Fibonacci number contributes only one leading digit out of $\abs{\sigma_{10}(F_{n})} \sim n\log_{10}{\phi}$ total digits. Up to $F_{m}$, the leading digits contribute $m$ digits out of a total of $D(m) \deq \sum_{n=1}^{m}\abs{\sigma_{10}(F_{n})} \sim \frac{1}{2}m^{2}\log_{10}{\phi}$ digits, where $D(m)$ denotes the number of base-$10$ digits in the concatenation of $F_{1}, \ldots, F_{m}$. The Benford bias therefore affects a fraction $O(1/m)$ of all digits and is asymptotically negligible.

More generally, for any fixed position $j$ from the left, the $j$-th digit of $F_{n}$ (when $F_{n}$ has at least $j$ digits) is a function of the fractional part $f_{n} = \{\log_{10}{F_{n}}\}$ alone, namely $\floor{10^{j-1+f_{n}}} \bmod 10$, and $f_{n} = \{n\log_{10}{\phi} + c + \eta_{n}\}$ differs from the linear sequence $\{n\log_{10}{\phi} + c\}$ only by the exponentially small Binet correction $\eta_{n} = O(\phi^{-2n})$ made explicit in Section~\ref{subsec:deep}. The bias of the digit-extraction function away from $1/10$ for each digit value is $O(10^{-j})$, decaying exponentially with the position \cite{diaconis}. Summing over all fixed positions, and writing $C_{d}(m)$ for the number of occurrences of the digit $d$ among the first $D(m)$ digits of $\FF$, the total structural bias in the frequency of $d$ is
\begin{equation}
\abs{\frac{C_{d}(m)}{D(m)} - \frac{1}{10}}_{\text{structural}} \leq \frac{\sum\limits_{j=1}^{\infty} O(10^{-j}) \cdot m}{\frac{1}{2}m^{2}\log_{10}{\phi}} = O(1/m).
\end{equation}

\subsection{The trailing digits}
The trailing digits of $F_{n}$ are governed by $F_{n} \bmod 10^{k}$, which, by Wall's theorem \cite{wall}, is purely periodic with Pisano period $\pi(10^{k})$, where $\pi(m)$ denotes the least integer $P \geq 1$ such that $F_{n+P} \equiv F_{n} \pmod{m}$ for all $n \geq 1$; one has $\pi(10^{k}) \asymp 10^{k}$. The distribution of these trailing blocks over one period is non-uniform for every $k$: by a theorem of Kuipers and Shiue \cite{kuipers}, with the general second-order case treated by Bumby \cite{bumby}, the Fibonacci sequence is uniformly distributed modulo $m$ if and only if $m$ is a power of $5$, and $10^{k} = 2^{k}5^{k}$ is never such a power. The obstruction is entirely $2$-adic: modulo $5^{k}$ the sequence is equidistributed \cite{niederreiter}, whereas modulo $2$ it runs $1, 1, 0, 1, 1, 0, \ldots$ with period $3$, so that $F_{n}$ is even precisely when $3 \mid n$ and odd values outnumber even ones two to one.

For $k = 1$ the period is $\pi(10) = 60$, and this parity bias renders the last digit non-uniform in a completely explicit way: over one period each odd digit $\{1,3,5,7,9\}$ occurs $8$ times and each even digit $\{0,2,4,6,8\}$ occurs $4$ times, that is, with frequencies $2/15$ and $1/15$ respectively \cite{spilker}. The same parity mechanism biases the trailing $k$-blocks for every $k$.

However, the same counting argument employed for the leading digits applies. The trailing $k$ digits of each $F_{n}$ contribute $k$ digits per Fibonacci number, for a total of $km$ out of $D(m) \sim \frac{1}{2}m^{2}\log_{10}{\phi}$ digits. For any fixed $k$ this fraction is $O(k/m) \to 0$, so the Pisano bias of the trailing digits, like the Benford bias of the leading digits, affects an asymptotically negligible fraction of $\FF$.

\subsection{The deep digits}
\label{subsec:deep}
The leading and trailing digits each contribute a vanishing fraction of $\FF$. The bulk of the digits occupy ``deep'' positions: writing $L_{n} \deq \abs{\sigma_{10}(F_{n})} \sim n\log_{10}{\phi}$ for the digit length of $F_{n}$, the positions $j \leq \gamma L_{n}$ account for only a fraction $\gamma$ of the digits of $F_{n}$, for any fixed $0 < \gamma < 1$, so the typical digit of $\FF$ sits at a depth comparable to the length of the Fibonacci number containing it. For such positions the counting arguments of the previous subsections give nothing, and the natural tool is equidistribution combined with a discrepancy bound. We set this up at a fixed depth $j$ and a fixed cutoff $m$.

The starting point is an exact digit-extraction identity. For $1 \leq j \leq L_{n}$, the $j$-th digit of $F_{n}$, counted from the left, equals $\floor{10^{j-1+f_{n}}} \bmod 10$, where $f_{n} \deq \{\log_{10}{F_{n}}\}$; indeed $F_{n} = 10^{L_{n}-1+f_{n}}$ by the definition of $L_{n}$ and $f_{n}$. Moreover, the Binet formula $F_{n} = (\phi^{n}/\sqrt{5})(1-(-1)^{n}\phi^{-2n})$ gives $\log_{10}{F_{n}} = n\log_{10}{\phi} + c + \eta_{n}$, with $c = -\frac{1}{2}\log_{10}{5}$ as before and $\abs{\eta_{n}} \leq \phi^{-2n}$, so the parameter $f_{n}$ is an exponentially small perturbation of the linear equidistributed sequence $\{n\alpha + c\}$, where $\alpha = \log_{10}{\phi}$. No approximation is involved in the identity itself; the rounding in $F_{n} = \floor{\phi^{n}/\sqrt{5} + \frac{1}{2}}$ enters only through $\eta_{n}$ and, as we will see, is harmless at every depth.

Fix a depth $j \geq 2$ (the case $j = 1$ is the Benford analysis above) and a cutoff $m$, let 
\[
n_{j} \deq \min\{n \colon L_{n} \geq j\} \sim j/\log_{10}{\phi}
\]
be the first index whose Fibonacci number reaches depth $j$, and put $M \deq m - n_{j} + 1$. The empirical frequency of the digit $d$ at depth $j$ across $F_{n_{j}}, \ldots, F_{m}$ is
\begin{equation}
\nu_{d}(j;m) \deq \frac{1}{M} \sum_{n=n_{j}}^{m} \one[\floor{10^{j-1+f_{n}}} \equiv d \mkern-12mu\pmod{10}] = \frac{1}{M} \sum_{n=n_{j}}^{m} g_{j,d}(f_{n}),
\end{equation}
where 
\[
g_{j,d}(f) \deq \one[\floor{10^{j-1+f}} \equiv d \mkern-12mu\pmod{10}]
\]
for $f \in [0,1)$. As $f$ increases over $[0,1)$, $\floor{10^{j-1+f}}$ steps through the integers $k = 10^{j-1}, \ldots, 10^{j}-1$, dwelling on each for an interval of length $\log_{10}(1+1/k) \asymp 10^{-j}$; hence $g_{j,d}$ is the indicator of a disjoint union of $9 \cdot 10^{j-2}$ such intervals, one for each $k \equiv d \pmod{10}$, with mean $\int_{0}^{1} g_{j,d}(f)\,\df = \frac{1}{10} + O(10^{-j})$, the positional Benford bias of \cite{diaconis}, and total variation $V(g_{j,d}) = 18 \cdot 10^{j-2} + O(1)$, exponentially large in the depth. Koksma's inequality (the one-dimensional Koksma--Hlawka inequality) \cite[Chap.~2, Thm.~5.1]{kuipers-niederreiter} then gives
\begin{equation}
\label{eq:kh}
\abs{\nu_{d}(j;m) - \int_{0}^{1}g_{j,d}(f)\,\df\,} \leq V(g_{j,d}) \cdot D_{M}^{*},
\end{equation}
where $D_{M}^{*}$ is the star discrepancy of the $M$ parameters $f_{n_{j}}, \ldots, f_{m}$.

Two remarks are in order before \eqref{eq:kh} is put to use. First, the Binet correction is negligible uniformly in $j$: splitting the range at $n_{0} = \lceil \log_{\phi}{m} \rceil$, the at most $n_{0}$ initial points alter the counting proportions that define $D_{M}^{*}$ by $O((\log{m})/M)$, while for $n > n_{0}$ each point is shifted by $\abs{\eta_{n}} \leq \phi^{-2n_{0}} \leq m^{-2}$, and a uniform shift of size $\delta$ changes the star discrepancy by at most $\delta$; hence the discrepancy of the $f_{n}$ exceeds that of the linear sequence $\{n\alpha + c\}$ by $O((\log{m})/m)$, at every depth. The depth $j$ never enters this comparison, which disposes of the rounding issue. Second, for the linear sequence Weyl's theorem gives $D_{M}^{*} = o(1)$ unconditionally, while sharp polynomial rates of the form $D_{M}^{*} = O((\log{M})^{2}/M)$ hold under Diophantine assumptions on $\alpha$ (e.g.\ for badly approximable or finite-type $\alpha$) and are not known unconditionally for $\alpha = \log_{10}{\phi}$.

What \eqref{eq:kh} delivers, and where it fails, can now be read off. At any fixed depth $j$, the unconditional $D_{M}^{*} = o(1)$ already yields $\nu_{d}(j;m) \to \frac{1}{10} + O(10^{-j})$ as $m \to \infty$: every fixed position is asymptotically uniform up to its exponentially small positional bias, in agreement with the leading-digits analysis. The bound, however, degrades exponentially with the depth. Even granting the optimistic conditional rate, the right-hand side of \eqref{eq:kh} is of order $10^{j}(\log{m})^{2}/m$, which tends to $0$ only in the range $10^{j} = o(m/(\log{m})^{2})$, that is, for depths $j \leq (1-o(1))\log_{10}{m}$, and becomes vacuous, larger than the trivial bound $\abs{\nu_{d} - \frac{1}{10}} \leq 1$, as soon as $10^{j}$ reaches $m/(\log{m})^{2}$. This horizon is intrinsic to the method and not a reflection of missing Diophantine information about $\log_{10}{\phi}$: every set of $M$ points in $[0,1)$ has star discrepancy $D_{M}^{*} \geq 1/(2M)$, so the right-hand side of \eqref{eq:kh} exceeds $10^{j-1}/(2m)$ for every parameter sequence whatsoever, and once $j - \log_{10}{m} \to \infty$ the inequality is vacuous even against a hypothetical sequence of optimal discrepancy. No bounded-variation discrepancy argument can see past depth $\asymp \log_{10}{m}$.

This horizon falls exponentially short of the deep digits. The positions at depths $j = O(\log{m})$ across all of $F_{1}, \ldots, F_{m}$ number $O(m\log{m})$ out of $D(m) \sim \frac{1}{2}m^{2}\log_{10}{\phi}$, an $O((\log{m})/m)$ fraction, while the typical digit of $F_{n}$ sits at depth proportional to $L_{n} \asymp n$, beyond the horizon for all but the first $O(\log{m})$ Fibonacci numbers; at such depths $V(g_{j,d}) \asymp 10^{j}$ is exponential in $n$, while no discrepancy, however favorable, decays faster than polynomially in $m$. The place-value analysis of Section~\ref{subsec:benfield-manes} will encounter exactly the same $O(m\log{m})$-digit horizon from the opposite, least-significant side. We also note that, even within its horizon, \eqref{eq:kh} controls the frequencies $\nu_{d}(j;m)$ aggregated over $n$, whereas condition~\eqref{eq:condition} demands digit statistics within each individual $F_{n}$; converting depth-by-depth averages into almost-all-$n$ statements would require second-moment information, on correlations across depths within a single $F_{n}$, that discrepancy estimates do not supply.

We emphasize that this is a failure of the Koksma--Hlawka discrepancy route, not an impossibility result. The argument shows that bounded-variation bounds combined with equidistribution of $\{n\log_{10}{\phi}\}$ cannot establish condition~\eqref{eq:condition}, and it leaves open the possibility that some altogether different technique could control the deep digits. The digit at a deep position within $F_{n}$ oscillates, as a function of the equidistributed parameter $f_{n}$, on a scale exponentially finer than any discrepancy bound can resolve. The problem is closely related to the open question of whether the digits of $\alpha^{n}$ for algebraic $\alpha > 1$ are normally distributed in base $10$, a question that remains unresolved even for $\alpha = 2$.

\subsection{The Benfield--Manes place-value approach}
\label{subsec:benfield-manes}

A different angle on the same problem was taken by Benfield and Manes \cite{benfield-manes}, who proposed to attack the normality of $\FF$ via the periodicity of the Fibonacci sequence modulo powers of the base. For an integer base $\beta \geq 2$ and $k \geq 0$, define
\begin{equation}
\label{eq:phi-place}
\Phi_{\beta^{k}}(n) \deq \floor{F_{n}/\beta^{k}} \bmod \beta,
\end{equation}
the digit at the $\beta^{k}$-place of $F_{n}$ in base $\beta$, with the convention $\Phi_{\beta^{k}}(n) = 0$ when $F_{n} < \beta^{k}$. By Wall's theorem \cite{wall}, the Fibonacci sequence is purely periodic modulo every integer $m$, with Pisano period $\pi(m)$. Since the $\beta^{k}$-place digit depends only on $F_{n} \bmod \beta^{k+1}$, the sequence $(\Phi_{\beta^{k}}(n))_{n \in \NN}$ is itself periodic, with period dividing $\pi(\beta^{k+1})$. Combining this with classical residue-distribution results of Niederreiter \cite{niederreiter} for moduli $5^{k}$ and Jacobson \cite{jacobson} for moduli $2^{k}$ with $k \geq 5$ (and more generally for $5^{x}2^{y}$), Benfield and Manes show that for every base of the form $\beta = 5^{x}2^{y}$ each digit $d \in \{0,1,\ldots,\beta-1\}$ appears in $(\Phi_{\beta^{k}}(n))_{n=1}^{\pi(\beta^{k+1})}$ exactly $\pi(\beta^{k+1})/\beta$ times. In other words, the place-value sequence at every fixed depth is uniformly distributed over one Pisano period. From this they conclude that $\FF$ is normal in every base of the form $5^{x}2^{y}$.

We accept the Niederreiter--Jacobson--Wall machinery exactly as stated: for every base $\beta = 5^{x}2^{y}$ and every depth $k \geq 0$, the place-value sequence $(\Phi_{\beta^{k}}(n))_{n \in \NN}$ is periodic with period dividing $\pi(\beta^{k+1})$ and each digit $d \in \{0,1,\ldots,\beta-1\}$ appears exactly $\pi(\beta^{k+1})/\beta$ times per period. What does not follow, in our view, is the inference from this per-period uniformity of each fixed place-value sequence to normality of the concatenated constant $\FF$. The obstruction is best appreciated by viewing the digits of $F_{1}, F_{2}, F_{3}, \ldots$ as the entries of an infinite ragged array whose $n$-th column holds the base-$\beta$ digits of $F_{n}$, written from the most significant digit at the top down to the units digit at the bottom and aligned along this common bottom (units) row, so that row $k \geq 0$, counted upward from the bottom, holds the $\beta^{k}$-place digits $\Phi_{\beta^{k}}(n)$. The concatenation $\FF$ reads this array {column by column}, taking each column from top (most significant) to bottom (least significant). There are then two distinct natural senses in which one might ask for ``uniform digit distribution'':
\begin{itemize}
\item[--] \emph{Row uniformity:} fix a row $k$ and let $n$ vary. This is what the Niederreiter--Jacobson results, fed through Wall periodicity, deliver in the bases $5^{x}2^{y}$.
\item[--] \emph{Column uniformity:} fix a column $n$ (large) and let the row index vary within that column. This is the requirement that the digit string of $F_{n}$ be $(\varepsilon,k)$-normal, that is, condition~($iii$) of Theorem~\ref{thm:criterion}.
\end{itemize}
The two senses are dual but logically independent. An infinite digit array can have all rows uniform without any column being uniform, and vice versa. The reading order of $\FF$ is the column order; the uniformity called upon by the sufficient criterion of Theorem~\ref{thm:criterion} is therefore column-wise rather than row-wise. We stress that this concerns one sufficient route only: column-wise $(\varepsilon,k)$-normality of almost all $F_{n}$ suffices for the normality of $\FF$, but it is not known to be necessary, and $\FF$ could in principle be normal without it.

To illustrate the gap between the two senses, the following example exhibits a ragged digit array with the strongest possible form of row uniformity, to wit, exact periodic balance of the kind delivered by the Niederreiter--Jacobson--Wall machinery, whose column-wise concatenation is simply normal but fails already at block length $2$.

\begin{example}[Row uniformity does not imply column-wise normality]
\label{ex:counterexample}
Fix base $\beta = 10$. Define a ragged digit array by specifying, for each $n \geq 1$, a column length $\ell_{n} = n$ and a type $t_{n} \in \{0,1,\dots,9\}$ with $t_{n} \equiv n \!\pmod{10}$, and declare the $n$-th column to be the digit string $t_{n}\,t_{n}\,\cdots\,t_{n}$ containing $\ell_{n}$ copies of $t_{n}$. In the row/column convention of this subsection, row $k$ then consists of the $k$-th place-value digit of each column long enough to have one, namely
\begin{equation}
(\Phi_k(n))_{n:\, \ell_{n}>k} \,=\, (t_{n})_{n>k} \,=\, t_{k+1}, t_{k+2}, t_{k+3}, \dots,
\end{equation}
which is a shift of the period-$10$ sequence $0, 1, 2, \dots, 9$. In particular, every row is purely periodic, with period $10$, and each digit
$d \in \{0,1,\dots,9\}$ appears exactly once per period at every depth $k$. This is precisely the kind of exact per-period balance furnished by the Niederreiter--Jacobson--Wall machinery for the Fibonacci digit array in bases
$5^{x}2^{y}$.

The concatenation read column by column is
\begin{equation}
C \deq t_{1} (t_{2}\,t_{2}) (t_{3}\,t_{3}\,t_{3}) (t_{4}\,t_{4}\,t_{4}\,t_{4})\cdots.
\end{equation}
Concretely, since $t_{n} \,\equiv\, n \!\pmod{10}$ and $\ell_{n} = n$, constant $C = 0.122333444455555\cdots$, a single $1$, then two $2$s, three $3$s, four $4$s, and so on, with the tenth column contributing ten $0$s. Thus digit $d$ appears in maximal runs of length $\ell_{n}$ whenever $t_{n}=d$. Single-digit frequencies converge to $1/10$. If
\begin{equation}
L_{N} \deq \sum_{n \leq N} \ell_{n} = \sum_{n \leq N} n,
\end{equation}
then for each $d \in \{0, \dots, 9\}$
\begin{equation}
\sum_{\substack{n \leq N \\ t_{n}=d}} \ell_{n} = \frac{1}{10}L_{N} + O(N),
\end{equation}
because the indices in each residue class modulo $10$ contribute one tenth of the sum $\sum\limits_{n \leq N} n$ up to an $O(N)$ error. Since $L_{N} \sim N^{2}/2$, it follows that
\begin{equation}
\lim_{N \to \infty} \frac{\sum\limits_{\substack{n \leq N \\ t_{n}=d}} \ell_{n}}{L_{N}}
= \frac{1}{10} \quad \text{for every } d\in\{0,\dots,9\}.
\end{equation}
Thus the column-wise concatenation is simply normal.

At depth $2$, however, the situation is starkly non-uniform. The $2$-block $dd$ occurs $\ell_{n}-1$ times inside every column of type $t_{n}=d$, so
\begin{equation}
\lim_{N \to \infty}
\frac{\#\{(d,d)\text{-blocks in the first }L_{N}\text{ digits of }C\}}{L_{N}-1}
= \lim_{N \to \infty} \frac{\sum\limits_{\substack{n \leq N \\ t_{n}=d}}(\ell_{n}-1)}{L_{N}-1} = \frac{1}{10}.
\end{equation}
By contrast, if $d\ne e$, then the block $de$ can arise only at a boundary between two consecutive columns, and hence contributes at most one occurrence per boundary. Therefore
\begin{equation}
\frac{\#\{(d,e)\text{-blocks in the first }L_{N}\text{ digits of }C\}}{L_{N}-1}
\leq
\frac{N-1}{L_{N}-1}
\longrightarrow 0.
\end{equation}
The $2$-block distribution of $C$ is therefore concentrated on the diagonal $\{(d,d)\colon d \in \{0, \dots, 9\}\}$, with mass $1/10$ on each diagonal block and mass $0$ on every off-diagonal block, instead of the uniform value $1/100$ required for normality at block length $2$.

In summary, in the above construction every row of the array is perfectly balanced, with exact periodic uniformity at every depth, yet the column-wise concatenation is simply normal and still fails already at block length $2$. Row uniformity and column-wise normality are logically independent properties of a ragged digit array, and no purely row-based argument can suffice to establish the latter.
\end{example}

Example~\ref{ex:counterexample} should be understood only as a counterexample to the inference from fixed-depth row uniformity to column-wise normality in a given base. By itself, it does not address the stronger (Pillai-based) route based on simple normality in all powers of the base.

To see why row uniformity does not imply normality of $\FF$, fix a base $\beta = 5^{x}2^{y}$ and consider the partial concatenation of $F_{1}, \ldots, F_{m}$. The contribution of row $k$ to this partial concatenation consists of those $n \in \{1,\ldots,m\}$ with $F_{n} \geq \beta^{k}$, namely $n_{k} \leq n \leq m$ with $n_{k} \sim k\log\beta/\log\phi$. The fact that $(\Phi_{\beta^{k}}(n))$ contains each digit equally often {over one full Pisano period of length $\pi(\beta^{k+1}) \asymp \beta^{k+1}$} is informative for the partial sum only when $m - n_{k} \gg \pi(\beta^{k+1})$, i.e., when $k \ll \log_{\beta} m$. The rows where the per-period averaging has effectively occurred therefore contribute only $O(m\log m)$ digits to the partial concatenation, out of a total of $D(m) \sim \frac{1}{2}m^{2}\log\phi/\log\beta$. The ``averaged'' fraction of the digits of $\FF$ at stage $m$ is thus $O((\log m)/m) \to 0$, while the bulk of the digits live in rows $k \in [\log_{\beta}m, \, m\log\phi/\log\beta]$, where the period-averaging argument is silent. These ``deep but not deep enough'' rows are precisely the digits at large place values within large Fibonacci numbers, which is exactly the obstacle isolated in Section~\ref{subsec:deep} above.

The two perspectives can therefore be summarized as follows. The Benfield--Manes approach establishes {row} uniformity of the digit array of $F_{n}$ in the bases $5^{x}2^{y}$; the Pollack--Vandehey criterion of Theorem~\ref{thm:criterion} demands column uniformity in the form of $(\varepsilon,k)$-normality of almost all Fibonacci numbers; and neither implies the other. In particular, the bases $5^{x}2^{y}$ that are most favorable for the row-side analysis enjoy no special status from the column side. Although $(\varepsilon,k)$-normality is itself a base-$b$ notion and condition~\eqref{eq:condition} is stated here in base $10$, the corresponding column-side obstruction---namely, asking whether almost all $F_{n}$ have $(\varepsilon,k)$-uniform digit strings in the chosen base---arises in every base, and the special bases $5^{x}2^{y}$ do not appear to simplify it. Condition~\eqref{eq:condition} and its analogues in other bases remain, in our view, the natural route to the normality of $\FF$ via the sufficient criterion of Theorem~\ref{thm:criterion}, and the principal obstacle to carrying that route through.


\section{Numerical experiments}
\label{sec:experiments}

To collect empirical evidence, we compute finite prefixes of $\FF$ by concatenating $F_{1}, F_{2}, \ldots, F_{N}$ for large $N$ and measure:

\begin{enumerate}
\item[($i$)] \emph{Single-digit frequencies}: the count $C_{d}(N)$ of each digit $d \in \{0,\ldots,9\}$ in the first $D(N) = \sum_{n=1}^{N}\abs{\sigma_{10}(F_{n})}$ digits of $\FF$, and the deviation $C_{d}(N)/D(N)-1/10$.

\item[($ii$)] \emph{Block frequencies}: the count of each block of length $k = 2, 3$, and $4$ in $\FF$, including cross-boundary blocks between consecutive Fibonacci numbers, and the deviation from $10^{-k}$.

\item[($iii$)] \emph{Individual $(\varepsilon,k)$-normality}: for each $F_{n}$, the maximum deviation $\delta_{n,k} = \max_{w \in \{0, \ldots, 9\}^{k}} \abs{N(w,\sigma_{10}(F_{n}))/\abs{\sigma_{10}(F_{n})}-10^{-k}}$, and the fraction of $n \leq N$ for which $\delta_{n,k} > \varepsilon$ as a function of $N$.
\end{enumerate}

\noindent All computations were carried out in Python~3.12 using the language's native arbitrary-precision integers, with no external libraries. The digit statistics are accumulated in a single streaming pass in which the concatenation is never materialized: each $F_{n}$ is obtained from its predecessors by one arbitrary-precision addition, converted once to its base-$b$ digit string, and folded into the running digit and $k$-block counters, so that peak memory is that of a single Fibonacci number plus the $b^{k}$ counter tables, rather than that of the $2.6 \times 10^{10}$-digit stream. Blocks straddling the junction between consecutive Fibonacci numbers are counted exactly once by carrying the last $k_{\max} - 1$ digits of the stream from one $F_{n}$ to the next, and the positional classification of Section~\ref{subsec:decompose10} is performed in the same pass, so the total cost is $O(k_{\max})$ character operations per digit of the concatenation. The subquadratic \texttt{int}-to-\texttt{str} conversion introduced in Python~3.11 is essential both for the streaming conversion and for the per-$F_{n}$ diagnostics, since the largest Fibonacci numbers in our run have $\lfloor 500\,000\log_{10}{\phi} \rfloor + 1 = 104\,494$ decimal digits. The full base-$10$ run up to $N = 500{,}000$, with block analysis to $k = 4$ and positional decomposition, required approximately $40\,500$~seconds ($\approx 11$\,h) of wall time on an Apple M1 Pro chip. The corresponding base-$2$ run took approximately $132\,400$~seconds ($\approx 37$\,h), roughly $3.3$ times longer, in line with the fact that the base-$2$ concatenation has roughly $3.3$ times more digits and the block-counting cost is essentially linear in the symbol count for both bases. The complete source code, the original output files from which every table in this paper is drawn, and a post-processing script for Good's serial statistic are openly available; see the data and code availability statement at the end of the paper.

\begin{remark}[Interpretation of the statistical tests]
\label{rem:stats}
The chi-squared statistics, $p$-values, and $z$-scores reported below are computed under the null hypothesis that the observed digits are independent draws from the uniform distribution on $\{0,\ldots,b-1\}$. For a deterministic concatenation with overlapping blocks, this \iid\ model is not formally justified, and the reported $p$-values should be read as heuristic diagnostics rather than as tests of a well-specified null. Likewise, the $D^{-1/2}$ benchmark we compare against below is the fluctuation scale of an \iid\ uniform sequence, not a consequence of normality per~se: a normal number need not exhibit discrepancy decaying at that rate. The observed digit statistics are, however, compatible with \iid-like fluctuations at the tested scales. We view this as heuristic evidence, not a proof of normality.
\end{remark}

\subsection{Single-digit frequencies}
We computed the concatenation of $F_{1}, \ldots, F_{N}$ for $N = 500{,}000$, producing $D(N) = 26{,}123{,}582{,}538$ base-$10$ digits, or about $26.1$~GB of decimal text at one byte per digit. Notice, however, that since the computation is streaming, these digits are never written to disk. The observed single-digit frequencies are shown in Table~\ref{tab:k1}. The maximum deviation from $1/10$ is $3.43 \times 10^{-6}$, and the chi-squared statistic is $\chi^{2} = 7.48$ on $9$ degrees of freedom, giving a $p$-value of $0.587$. The null hypothesis of uniform digit distribution cannot be rejected.

\begin{table}[h]
\centering
\small
\begin{tabular}{ccccc}
\hline \hline
Digit $d$ & Count $C_{d}(N)$ & Frequency & Deviation $\times 10^{6}$ & $z$-score \\
\hline
$0$ & $2\,612\,318\,631$ & $0.09999848$ & $-1.52$ & $-0.817$ \\
$1$ & $2\,612\,447\,898$ & $0.10000343$ & $+3.43$ & $+1.849$ \\
$2$ & $2\,612\,361\,314$ & $0.10000012$ & $+0.12$ & $+0.063$ \\
$3$ & $2\,612\,407\,108$ & $0.10000187$ & $+1.87$ & $+1.008$ \\
$4$ & $2\,612\,357\,524$ & $0.09999997$ & $-0.03$ & $-0.015$ \\
$5$ & $2\,612\,396\,862$ & $0.10000148$ & $+1.48$ & $+0.796$ \\
$6$ & $2\,612\,316\,010$ & $0.09999838$ & $-1.62$ & $-0.871$ \\
$7$ & $2\,612\,304\,937$ & $0.09999796$ & $-2.04$ & $-1.100$ \\
$8$ & $2\,612\,351\,037$ & $0.09999972$ & $-0.28$ & $-0.149$ \\
$9$ & $2\,612\,321\,217$ & $0.09999858$ & $-1.42$ & $-0.764$ \\
\hline \hline
\end{tabular} \\[6pt]
\caption{Single-digit frequencies in the first $D(N) = 26{,}123{,}582{,}538$ digits of $\FF$ ($N = 500{,}000$). The chi-squared test yields $\chi^{2} = 7.48$, $\df = 9$, $p = 0.587$.}
\label{tab:k1}
\end{table}

\subsection{Block frequencies}
Although normality is a statement about all block lengths $k \geq 1$ simultaneously, the tests we can carry out empirically are limited to small $k$. At $N = 500{,}000$ the expected count per $k$-block is $D(N)/10^{k} \approx 2.6 \times 10^{10-k}$, which is still large at $k = 6$ (about $2.6 \times 10^{4}$ per block) and drops below the level needed for a reliable full block-frequency test only around $k = 9$. The range $k \leq 4$ is nevertheless informative in the present problem, because the structural sources of bias identified in Section~\ref{sec:structural}, namely, the Benford distribution of leading digits and the Pisano periodicity of trailing digits, act at $k = 1$ and are coupled to each other at $k = 2$. Any gross failure of simple normality in $\FF$ would show up first at the smallest block lengths, and the tests below therefore probe the most likely avenues of structured deviation.

Block counts are collected with a sliding window of stride~$1$, including cross-boundary blocks. Because sliding-window $k$-block counts in an i.i.d.\ sequence are not multinomial---blocks with self-overlaps (e.g.\ ``$00$'' at shift~$1$) have larger variance than those without---the naive Pearson statistic with $b^{k}-1$ degrees of freedom is not $\chi^{2}$-distributed under the uniform null. The correct reference distribution for sliding blocks is Good's serial statistic \cite{good-1953},
\begin{equation}
\label{eq:good}
\Delta\chi^{2}_{k} \,\deq\, \chi^{2}_{\,\text{naive}}(k) - \chi^{2}_{\,\text{naive}}(k-1),
\end{equation}
which is asymptotically $\chi^{2}$ with $b^{k}-b^{k-1}$ degrees of freedom. We therefore report both the naive and the Good-corrected $p$-values; the latter are the formally valid reference values under the \iid\ sliding-window null.

For $k = 2$, all $100$ two-digit blocks were counted. The largest absolute deviation from $10^{-2}$ is $2.29 \times 10^{-6}$, achieved by the block~``$12$''; the mean absolute deviation is $4.9 \times 10^{-7}$. The naive chi-squared is $\chi^{2} = 102.0$ on $99$ degrees of freedom (naive $p = 0.398$). Good's corrected statistic is $\Delta\chi^{2}_{2} = 94.5$ on $90$ degrees of freedom ($p = 0.352$), not significant at the $5\%$ level.

For $k = 3$, all $1000$ three-digit blocks were counted. The largest deviation is $6.2 \times 10^{-7}$ (block~``$089$''), with naive $\chi^{2} = 1051.9$ on $999$ degrees of freedom (naive $p = 0.119$) and Good $\Delta\chi^{2}_{3} = 949.9$ on $900$ degrees of freedom ($p = 0.121$).

For $k = 4$, all $10\,000$ four-digit blocks were counted. The largest deviation is $2.8 \times 10^{-7}$ (block~``$9211$''), with naive $\chi^{2} = 10\,198.0$ on $9999$ degrees of freedom (naive $p = 0.080$) and Good $\Delta\chi^{2}_{4} = 9146.1$ on $9000$ degrees of freedom ($p = 0.138$). No individual block exceeds the two-sided Bonferroni-adjusted critical value $\abs{z} \approx 4.565$ for $10\,000$ tests at family-wise level $5\%$, the largest observed $\abs{z}$-score being $4.51$. Both the $k=3$ and the $k=4$ block-frequency tests are consistent with the conjecture that $\FF$ is normal at these block lengths.

\subsection{Convergence rate}
Table~\ref{tab:evolution} shows the evolution of the maximum single-digit deviation as $N$ increases. A log-log regression of the maximum deviation on $D(N)$ yields
\begin{equation}
\max_{d} \abs{\frac{C_{d}(N)}{D(N)} - \frac{1}{10}} \approx 1.030 \cdot D(N)^{-0.5152}, \quad R^{2} = 0.9954,
\end{equation}
over seven orders of magnitude in $D(N)$. This regression uses the five sizes with $N \geq 100$; the smallest point ($N = 10$, only $14$ digits) is pre-asymptotic and is excluded. The fitted exponent is $0.515 \pm 0.020$ (standard error), consistent with the value $1/2$ characteristic of \iid\ uniform fluctuations, from which it differs by less than one standard error. Five points cannot pin the exponent down to better than about $4\%$, but the data give no indication of a systematic departure from the $D^{-1/2}$ rate. We emphasize, in line with Remark~\ref{rem:stats}, that the $D^{-1/2}$ benchmark is the \iid\ fluctuation scale and is not implied by normality per~se; what the data support is that the observed decay is compatible with \iid-like behavior over the scales tested.

\begin{table}[h]
\centering
\small
\begin{tabular}{rrccc}
\hline \hline
$N$ & $D(N)$ & $\max_{d}\abs{\text{dev}}$ & $\chi^{2}(k=1)$ & $p$-value \\
\hline
      $10$ &                $14$ & $1.86 \times 10^{-1}$ & $14.57$ & $0.103$ \\
     $100$ &            $1\,071$ & $2.98 \times 10^{-2}$ & $15.38$ & $0.081$ \\
    $1000$ &          $104\,750$ & $2.11 \times 10^{-3}$ & $10.50$ & $0.311$ \\
 $10\,000$ &      $10\,451\,934$ & $2.87 \times 10^{-4}$ & $20.97$ & $0.013$ \\
$100\,000$ &  $1\,044\,963\,704$ & $3.13 \times 10^{-5}$ & $21.42$ & $0.011$ \\
$500\,000$ & $26\,123\,582\,538$ & $3.43 \times 10^{-6}$ &  $7.48$ & $0.587$ \\
\hline \hline
\end{tabular} \\[6pt]
\caption{Evolution of the maximum single-digit deviation and chi-squared statistic as $N$ grows. The deviation decays at a rate compatible with the $D^{-1/2}$ fluctuation scale of an \iid\ uniform sequence; see Remark~\ref{rem:stats}. The log-log regression reported in the text uses the five sizes with $N \geq 100$, excluding the pre-asymptotic $N = 10$.}
\label{tab:evolution}
\end{table}

\subsection{Individual Fibonacci numbers}
We also computed, for each $F_{n}$ with $n \leq N$ and $\abs{\sigma_{10}(F_{n})} \geq 10$, the maximum single-digit deviation $\delta_{n,1}$. The fraction of $n$ for which $\delta_{n,1}$ exceeds various thresholds is shown in Table~\ref{tab:individual}. For $\varepsilon = 0.05$, only $0.12\%$ of Fibonacci numbers fail to be $(\varepsilon,1)$-normal; for $\varepsilon = 0.02$, the fraction is $0.79\%$; and for $\varepsilon = 0.01$, still only $3.20\%$. These figures are consistent with condition~(iii) of Theorem~\ref{thm:criterion} holding for $k=1$.

\begin{table}[h]
\centering
\small
\begin{tabular}{ccc}
\hline \hline
$\varepsilon$ & $\#\{n \colon \delta_{n,1} > \varepsilon\}$ & Fraction \\
\hline
$0.050$ &      $612$ & $0.12\%$ \\
$0.020$ &   $3\,952$ & $0.79\%$ \\
$0.010$ &  $15\,978$ & $3.20\%$ \\
$0.005$ &  $64\,224$ & $12.85\%$ \\
$0.002$ & $348\,281$ & $69.66\%$ \\
\hline \hline
\end{tabular} \\[6pt]
\caption{Fraction of Fibonacci numbers $F_{n}$ among the $499{,}956$ with $n \leq 500{,}000$ and $\abs{\sigma_{10}(F_{n})} \geq 10$ whose maximum single-digit deviation exceeds $\varepsilon$.}
\label{tab:individual}
\end{table}

A complementary comparison is informative. For a number with $K$ \iid\ uniform digits in base $10$, the standardized deviations $Z_{d} = (C_{d} - K/10)/\sqrt{0.09K}$ are approximately $N(0,1)$ and pairwise correlated with correlation $-1/9$ (since $\sum_{d} C_{d} = K$). Ignoring this mild correlation, $\EE[\max_{1 \leq d \leq 10}\abs{Z_{d}}] \approx 1.86$, so the expected maximum single-digit deviation is approximately $1.86\sqrt{0.09/K}$. Taking the ratio of the observed deviation $\delta_{n,1}$ to this random baseline, averaged over the $499\,047$ Fibonacci numbers with at least $200$ digits, we obtain a mean ratio of $1.00 \pm 0.28$. In other words, the single-digit distributions of individual Fibonacci numbers are, on average, statistically indistinguishable from those of random numbers of the same length, a finding consistent with condition~(iii), though of course not proving it.

\subsection{Positional decomposition}
\label{subsec:decompose10}

The block-frequency tests above treat all $k$-blocks in the concatenation on an equal footing, regardless of where they sit relative to the individual Fibonacci numbers. A finer analysis is possible by classifying each $k$-block into one of four positional categories:
\begin{itemize}
\item[--] \emph{leading}: the first $k$-block of some $F_{n}$, starting at position $0$ within $F_{n}$;
\item[--] \emph{trailing}: the last $k$-block of some $F_{n}$, starting at position $L_{n} - k$, when distinct from the leading block;
\item[--] \emph{middle}: a block at a strictly interior position within some $F_{n}$;
\item[--] \emph{boundary}: a block straddling the junction between consecutive $F_{n}$ and $F_{n+1}$.
\end{itemize}
These four categories partition all $k$-blocks. Since each Fibonacci number contributes at most two non-middle blocks (leading and trailing) plus at most $k-1$ boundary blocks, the leading, trailing, and boundary blocks together account for $O(kN)$ out of $D(N) \sim \frac{1}{2}N^{2}\log_{b}\phi$ total blocks and are an asymptotically vanishing fraction for any fixed $k$.

In base $10$, the leading digits of $F_{n}$ follow Benford's law (digit $1$ appears with frequency $\log_{10}{2} \approx 0.301$, digit $9$ with frequency $\log_{10}(10/9) \approx 0.046$), while the trailing digits follow the Pisano periodicity modulo $10$ with period $60$: even digits appear with frequency $1/15$ and odd digits with frequency $2/15$ \cite{spilker}. Both biases are visible in the decomposition but are asymptotically negligible.

Table~\ref{tab:decompose10} reports Good's serial statistic~\eqref{eq:good} computed separately on the middle blocks and on the boundary blocks, for $k = 2, 3, 4$. For these category-restricted statistics the serial difference in~\eqref{eq:good} is taken within the category itself: for the middle category, $\Delta\chi^{2}_{k} = \chi^{2}_{\text{mid}}(k) - \chi^{2}_{\text{mid}}(k-1)$, where $\chi^{2}_{\text{mid}}(1)$ is the Pearson statistic of the interior digits (positions $1, \ldots, L_{n}-2$ within each $F_{n}$), which are precisely the symbols from which the middle blocks are drawn. The middle blocks, comprising over $99.99\%$ of all blocks, are fully consistent with uniformity at every block length: the interior digits give $\chi^{2}_{\text{mid}}(1) = 3.02$ on $9$ degrees of freedom ($p = 0.96$), and the serial statistics give $p = 0.64$ at $k=2$, $p = 0.20$ at $k=3$, and $p = 0.17$ at $k=4$. As in the base-$2$ decomposition of Section~\ref{subsec:decompose}, Good's serial reference law is formally exact for the full sliding-window counts under the \iid\ null, not for the filtered middle-block subset used here; since the base-$10$ middle blocks show no significant deviation at any $k$, this does not affect the conclusion, but the same caution applies in principle.

\begin{table}[h]
\centering
\small
\begin{tabular}{ccccccc}
\hline \hline
$k$ & Middle $\Delta\chi^{2}$ & $\df$ & Middle $p$ & Boundary $\Delta\chi^{2}$ & $\df$ & Boundary $p$ \\
\hline
$2$ &    $84.7$ &   $90$ & $0.638$ & --- &   $90$ & $< 10^{-6}$ \\
$3$ &     $935$ &  $900$ & $0.203$ & --- &  $900$ & $< 10^{-6}$ \\
$4$ &  $9\,126$ & $9000$ & $0.174$ & --- & $9000$ & $< 10^{-6}$ \\
\hline \hline
\end{tabular} \\[6pt]
\caption{Positional decomposition of the base-$10$ block-frequency test ($N = 500{,}000$) using Good's serial statistic~\eqref{eq:good} with degrees of freedom $b^{k}-b^{k-1}$. The middle blocks (interior to individual $F_{n}$, comprising $99.99\%$ of all blocks) are consistent with uniformity at every $k$. The boundary blocks ($< 0.006\%$) are massively non-uniform under any sensible null. The precise $\Delta\chi^{2}$ is suppressed as the naive and Good statistics both yield $p < 10^{-6}$.}
\label{tab:decompose10}
\end{table}

The boundary chi-squared values are enormous because the junction $F_{n} \,|\, F_{n+1}$ juxtaposes a trailing digit governed by the Pisano period (odd digits favored $2:1$) with a leading digit governed by Benford's law (digit $1$ at frequency $\approx 0.30$). Among the boundary $2$-blocks, ``$91$'', ``$31$'', ``$71$'', ``$51$'', and ``$11$'' are the most overrepresented, all reflecting the Benford bias of the next Fibonacci number. This deterministic pattern propagates to longer blocks in a predictable way.

\subsection{Summary of base-$10$ experiments}
The numerical evidence across all four tests---single-digit frequencies, block frequencies, individual $(\varepsilon,1)$-normality, and positional decomposition---is consistent with the conjecture that $\FF$ is simply normal in base $10$, and is consistent with normality at block lengths $k = 2, 3$, and $4$. The observed convergence rate is, within the precision afforded by the data, indistinguishable from the $D^{-1/2}$ fluctuation scale of an \iid\ uniform sequence; see, however, Remark~\ref{rem:stats}. The positional decomposition shows that the pooled interior blocks, aggregated across all Fibonacci numbers, are statistically consistent with uniformity, and the visible structured deviation is concentrated at the boundary between consecutive Fibonacci numbers, an artifact that becomes negligible as $N \to \infty$. This is an aggregate statement and does not imply that each individual $F_{n}$ has uniform interior digits.


\section{Base-$2$ experiments}
\label{sec:base2}

We repeated the full suite of experiments in base $2$, concatenating the binary representations of $F_{1}, \ldots, F_{N}$ for $N = 500{,}000$, producing $D(N) = 86{,}780{,}082{,}284$ binary digits, or about $87$~GB if stored as text with one byte per binary digit (although we never actually write all bits down), roughly $\log_{2}{10} \simeq 3.3$ times more digits than the base-$10$ concatenation for the same $N$.

\subsection{Global digit and block frequencies}
Table~\ref{tab:k1b2} shows the single-bit frequencies. The split is $50.000006\%$\,/\,$49.999994\%$, with $\chi^{2} = 0.11$ on $1$ degree of freedom ($p = 0.742$). For blocks of size $k = 2, 3, 4$, collected with a sliding window as in Section~\ref{subsec:decompose10}, the naive Pearson statistics on $2^{k}-1$ degrees of freedom are $2.06$ (naive $p = 0.560$), $7.71$ (naive $p = 0.359$), and $18.56$ (naive $p = 0.234$). The corresponding Good-corrected serial statistics~\eqref{eq:good} on $2^{k-1}$ degrees of freedom are $\Delta\chi^{2}_{2} = 1.95$ ($p = 0.377$), $\Delta\chi^{2}_{3} = 5.65$ ($p = 0.227$), and $\Delta\chi^{2}_{4} = 10.85$ ($p = 0.210$). Under either reference law the global block-frequency tests are well within the expected range under the null hypothesis of uniformity.

\begin{table}[h]
\centering
\small
\begin{tabular}{ccccc}
\hline \hline
Digit $d$ & Count $C_{d}(N)$ & Frequency & Deviation $\times 10^{7}$ & $z$-score \\
\hline
$0$ & $43\,390\,089\,668$ & $0.50000056$ & $+5.6$ & $+0.329$ \\
$1$ & $43\,389\,992\,616$ & $0.49999944$ & $-5.6$ & $-0.329$ \\
\hline \hline
\end{tabular} \\[6pt]
\caption{Single-bit frequencies in the first $D(N) = 86{,}780{,}082{,}284$ bits of the base-$2$ concatenation ($N = 500{,}000$). The chi-squared test yields $\chi^{2} = 0.11$, $\df = 1$, $p = 0.742$.}
\label{tab:k1b2}
\end{table}

The convergence of the maximum single-digit deviation with $D$ is shown in Table~\ref{tab:evolutionb2}. A log-log regression over the five sizes with $N \geq 100$ (again excluding the pre-asymptotic $N = 10$) yields max$\abs{\text{dev}} \approx 0.725 \cdot D^{-0.563}$ with $R^{2} = 0.977$. As in the base-$10$ case, the fitted exponent is $0.563 \pm 0.050$ (standard error), compatible with the $D^{-1/2}$ rate to within the uncertainty of a five-point fit ($1/2$ lies about $1.3$ standard errors away).

\begin{table}[h]
\centering
\small
\begin{tabular}{rrccc}
\hline \hline
$N$ & $D(N)$ & $\max_{d}\abs{\text{dev}}$ & $\chi^{2}(k=1)$ & $p$-value \\
\hline 
      $10$ &                $34$ & $1.18 \times 10^{-1}$ & $1.88$ & $0.170$ \\
     $100$ &            $3\,442$ & $4.07 \times 10^{-3}$ & $0.23$ & $0.633$ \\
    $1000$ &          $346\,809$ & $1.35 \times 10^{-3}$ & $2.53$ & $0.112$ \\
 $10\,000$ &      $34\,708\,959$ & $4.42 \times 10^{-5}$ & $0.27$ & $0.602$ \\
$100\,000$ &  $3\,471\,178\,185$ & $1.90 \times 10^{-6}$ & $0.05$ & $0.823$ \\
$500\,000$ & $86\,780\,082\,284$ & $5.59 \times 10^{-7}$ & $0.11$ & $0.742$ \\
\hline \hline
\end{tabular} \\[6pt]
\caption{Evolution of the maximum single-digit deviation in base $2$. As in base $10$, the log-log regression in the text uses the five sizes with $N \geq 100$.}
\label{tab:evolutionb2}
\end{table}

\subsection{Positional decomposition}
\label{subsec:decompose}

We repeat the positional decomposition of Section~\ref{subsec:decompose10} in base $2$. Here every Fibonacci number $F_{n}$ has a leading bit of $1$, and for $n \geq 3$ the trailing bit follows the Pisano period modulo $2$, which has length $3$: the pattern is $1, 1, 0, 1, 1, 0, \ldots$, so two-thirds of all trailing bits are $1$. These forced bits create a deterministic pattern at every junction $F_{n} \,|\, F_{n+1}$ in the concatenation.

Table~\ref{tab:decompose} reports Good's serial statistic~\eqref{eq:good} computed separately on the middle blocks and on the boundary blocks, for $k = 2, 3, 4$, with the serial difference again taken within each category, as in Section~\ref{subsec:decompose10}.

\begin{table}[h]
\centering
\small
\begin{tabular}{ccccccc}
\hline \hline
$k$ & Middle $\Delta\chi^{2}$ & $\df$ & Middle $p$ & Boundary $\Delta\chi^{2}$ & $\df$ & Boundary $p$ \\
\hline
$2$ &  $5.4$ &  $2$ & $0.067$  & --- & $2$ & $< 10^{-6}$ \\
$3$ &  $5.9$ &  $4$ & $0.207$  & --- & $4$ & $< 10^{-6}$ \\
$4$ &  $5.8$ &  $8$ & $0.669$  & --- & $8$ & $< 10^{-6}$ \\
\hline \hline
\end{tabular} \\[6pt]
\caption{Positional decomposition of the base-$2$ block-frequency test ($N = 500{,}000$) using Good's serial statistic~\eqref{eq:good} with degrees of freedom $b^{k}-b^{k-1} = 2^{k-1}$, the serial difference being taken within each category. The middle blocks (interior to individual $F_n$, comprising 99.997\% of all blocks) show no significant serial structure at any tested $k$; the residual interior signal is a single-bit imbalance, discussed in the text. The boundary blocks ($< 0.002\%$) are massively non-uniform under any sensible null. The precise $\Delta\chi^{2}$ is suppressed as the naive and Good statistics both yield $p < 10^{-6}$.}
\label{tab:decompose}
\end{table}

The boundary chi-squared values are enormous because the junction $F_{n} \,|\, F_{n+1}$ always places a leading $1$ immediately after a trailing bit that is $1$ two-thirds of the time, forcing the $2$-block ``$11$'' to appear at frequency $\approx 2/3$ instead of $1/4$ among boundary blocks. For $k = 2$, the blocks ``$00$'' and ``$10$'' never occur, and the pattern propagates to longer blocks in a predictable way.

Under Good's corrected null, taken within the middle ensemble, the middle blocks show no significant serial structure at any tested block length ($\Delta\chi^{2}_{2} = 5.4$, $p = 0.067$; $\Delta\chi^{2}_{3} = 5.9$, $p = 0.207$; $\Delta\chi^{2}_{4} = 5.8$, $p = 0.669$). The residual interior signal sits instead at the single-bit level: the pooled interior bits split $0.500004\,/\,0.499996$, a microscopic excess of zeros of relative size $4.4 \times 10^{-6}$ that is nevertheless formally significant ($\chi^{2}_{\text{mid}}(1) = 6.72$ on $1$ degree of freedom, $p = 0.0095$), because with an aggregate sample of $D \sim 8.7 \times 10^{10}$ bits the test has massive power against even asymptotically vanishing structural constraints, such as the finite lengths of the individual excised interior strings. The absolute deviations in the middle-block frequencies remain less than $2 \times 10^{-6}$, comparable in magnitude to the base-$10$ deviations. Thus, while a strict \iid\ null for the interior bits is formally rejected at the single-bit level, the pooled interior digits are uniformly distributed for all practical purposes, no block-level structure appears beyond the single-bit imbalance, and essentially all of the structured macroscopic deviation is a boundary artifact that becomes negligible as $N \to \infty$. This is a statement about the aggregate distribution of interior digits and does not, by itself, imply uniformity of the interior digits of each individual $F_{n}$.

\subsection{Individual Fibonacci numbers in base $2$}
The individual $(\varepsilon,1)$-normality statistics in base $2$ are even better than in base $10$, as expected from the larger number of digits per $F_{n}$ ($\lfloor n\log_{2}{\phi}\rfloor + 1 \approx 1.44n$ versus $\lfloor n\log_{10}{\phi}\rfloor + 1 \approx 0.21n$). Table~\ref{tab:individualb2} reports the fraction of Fibonacci numbers exceeding various thresholds.

\begin{table}[h]
\centering
\small
\begin{tabular}{ccc}
\hline \hline
$\varepsilon$ & $\#\{n \colon \delta_{n,1} > \varepsilon\}$ & Fraction \\
\hline
$0.050$ &     $132$ &  $0.03\%$ \\
$0.020$ &     $916$ &  $0.18\%$ \\
$0.010$ &  $3\,657$ &  $0.73\%$ \\
$0.005$ & $14\,436$ &  $2.89\%$ \\
$0.002$ & $87\,039$ & $17.41\%$ \\
\hline \hline
\end{tabular} \\[6pt]
\caption{Fraction of Fibonacci numbers (among the $499\,986$ with $n \leq 500{,}000$ and at least $10$ binary digits) whose maximum single-digit deviation exceeds $\varepsilon$ in base $2$.}
\label{tab:individualb2}
\end{table}

Comparing with Table~\ref{tab:individual}, the fractions at each threshold are substantially smaller in base $2$; for instance, $0.03\%$ versus $0.12\%$ at $\varepsilon = 0.05$, and $2.89\%$ versus $12.85\%$ at $\varepsilon = 0.005$. This improvement is consistent with the larger digit count per $F_{n}$ in binary. For the ratio of observed to expected maximum deviation in base $2$, the multinomial constraint $C_{0} + C_{1} = K$ forces $\abs{C_{0}/K - 1/2} \equiv \abs{C_{1}/K - 1/2}$, so under the \iid\ model the expected maximum deviation reduces to $\EE[\abs{Z}]/(2\sqrt K) = \sqrt{2/\pi}\,/(2\sqrt K)$ with $Z \sim N(0,1)$. Against this baseline, the mean ratio over the $499\,712$ Fibonacci numbers with at least $200$ binary digits is $1.00 \pm 0.75$; the larger standard deviation reflects the degeneracy of the two-symbol case, where the ``max'' is a single folded normal rather than the max over several.

\subsection{Summary}
The base-$2$ and base-$10$ experiments are mutually consistent and point to the same conclusion: the concatenated Fibonacci constant $\FF$ is numerically consistent with the behavior expected of a normal number at the scales tested. The positional decomposition in both bases shows that the Benford-biased leading digits and the Pisano-periodic trailing digits of each $F_{n}$ create a deterministic boundary pattern that is massively non-uniform in isolation, but this boundary layer comprises less than $0.006\%$ of all blocks at $N = 500{,}000$---a fraction that is $O(1/N)$ and tends to $0$ as $N \to \infty$---and has no discernible effect on the overall digit distribution. 

In base $10$, the pooled middle blocks are fully consistent with uniformity at every block length tested. In base $2$, Good's serial statistics on the pooled middle blocks are not significant at any tested block length; the only formally significant interior signal is a single-bit imbalance of relative size $4.4 \times 10^{-6}$ ($p = 0.0095$), microscopic in absolute terms and detectable only because of the enormous aggregate sample, so it should be interpreted cautiously. Taken together, the pooled interior blocks are very close to uniform and the boundary effects account for essentially all of the visible structured deviation. We emphasize that this is an aggregate statement across all Fibonacci numbers; the experiments do not establish that each individual $F_{n}$ has uniform interior digits, which is the content of condition~(iii) and remains open. Numerical evidence cannot settle the question, which depends on the asymptotic behavior of $F_{n}$ for $n$ arbitrarily large. The data, however, provide no hint of the kind of persistent bias that would rule out normality.


\section{Concluding remarks}
\label{sec:concluding}

We have identified the $(\varepsilon,k)$-normality of almost all Fibonacci numbers as a sufficient condition for the normality of the Fibonacci constant $\FF$, and shown that this condition lies beyond the reach of current equidistribution techniques due to the exponential oscillation of the digit function at deep positions within $F_{n}$. The problem belongs to a broader class of open questions about the digit distribution of exponential sequences. Even the simplest instance of this problem, namely, whether the digits of $2^{n}$ in base $10$ are equidistributed for most $n$, remains unresolved, despite considerable attention \cite{bugeaud}. A proof of condition~\eqref{eq:condition}, even for $k=1$, which would yield simple normality of $\FF$, would represent significant progress in the area.

Our numerical experiments, carried out on the first $500{,}000$ Fibonacci numbers in both base $10$ and base $2$, yielding $2.6 \times 10^{10}$ decimal digits and $8.7 \times 10^{10}$ bits, respectively, showed the maximum single-digit deviation decaying at a rate indistinguishable from the $D^{-1/2}$ \iid\ benchmark, and the block-frequency statistics for $k = 2, 3, 4$ revealing no significant departure from uniformity in either base, both compatible with \iid-like fluctuations at the tested scales. A positional decomposition of the $k$-block counts further shows that the visible structured deviation is almost entirely concentrated at the junctions between consecutive $F_{n}$, a deterministic artifact driven by Benford-biased leading digits meeting Pisano-periodic trailing digits, while the pooled interior blocks are close to uniform (with deviations at or below the $10^{-6}$ level). In line with Remark~\ref{rem:stats}, we stress that these observations are heuristic diagnostics rather than formal evidence: the statistics are suggestive but not decisive, providing evidence compatible with normality but no proof of it, and, crucially, the aggregate uniformity of pooled interior blocks does not imply the per-$F_{n}$ uniformity required by condition~\eqref{eq:condition}. The experiments therefore reveal no obstruction to normality but do not, and cannot, resolve it.

Two possible directions for future work are the following. First, one could exploit the fact that multi-digit addition with carry propagation is known to have a ``mixing'' effect on digit distributions when the summands are random \cite{diaconis-fulman}. A possible approach is to investigate whether iterating the Fibonacci recurrence produces a similar effect for the deterministic Fibonacci sequence, gradually driving the digits of $F_n$, and thence of $\FF$, toward equidistribution, though making this intuition rigorous (likely by devising a higher-order Markov chain of some sort) remains an open challenge. Second, one could seek weaker structural results that fall short of full normality but still provide nontrivial information, for instance bounds on the finite-state dimension of $\FF$ or proofs that the digit frequencies lie in a prescribed interval. We are particularly motivated by the first possibility, which, even if it does not lead to a solution, may still provide partial or obstruction results.


\section*{Declarations}

\subsection*{Conflict of interest}

The author has no relevant financial or non-financial interests to disclose. 

\subsection*{Funding}

The author received partial financial support from Funda\c{c}\~{a}o de Amparo \`{a} Pesquisa do Estado de S\~{a}o Paulo -- FAPESP, Brazil, through grant no.~2020/04475-7. 

\subsection*{Use of generative AI and AI-assisted technologies}

The author benefited from the use of AI language models in the development of this work. Initially, Anthropic's Claude~4.6~Opus, via the Claude Code application, was used to translate and optimize legacy C code and bash scripts into Python~3, enabling highly efficient handling of large integers. During the theoretical development in Section~\ref{subsec:benfield-manes}, the same model was prompted to produce a counterexample, which after mathematical refinement by the author became Example~\ref{ex:counterexample}.

During the preparation of the submitted version of the manuscript, Google's Gemini~3.1~Pro identified a covariance issue with overlapping blocks in Sections~\ref{sec:experiments} and \ref{sec:base2} and suggested the use of Good's serial statistic, resulting in the recalculation of the statistics without materially affecting the conclusions. 

The AI models did not participate in the drafting of Sections~\ref{sec:intro}--\ref{sec:reduction}, the majority of Sections~\ref{sec:structural}--\ref{sec:base2}, and Section~\ref{sec:concluding}, or the selection of the bibliography. The author remains fully responsible for the conceptualization, methodology, data generation and analysis, writing, and rigorous verification of the submitted manuscript.

\subsection*{Data availability}

The Python source code used for all computations in this paper, the original output files from which every table in Sections~\ref{sec:experiments} and~\ref{sec:base2} is drawn, a post-processing script for Good's serial statistic, and a README mapping each table to the corresponding output are openly available on Zenodo at DOI \href{https://doi.org/10.5281/zenodo.21268947}{10.5281/zenodo.21268947}. The computations require only a standard CPython~$\geq 3.11$ interpreter, with no third-party dependencies.



\begin{thebibliography}{69}

\bibitem{bailey-crandall}
Bailey, D. H., Crandall, R. E. (2003). Random generators and normal numbers. \textit{Exp. Math.} {11}(4): 527--546. DOI: 10.1080/10586458.2002.10504704.

\bibitem{benfield-manes}
Benfield, B., Manes, M. (2022). The Fibonacci sequence is normal base 10. Preprint arXiv:2202.08986~[math.NT]. DOI: 10.48550/arXiv.2202.08986.

\bibitem{besicovitch}
Besicovitch, A. S. (1935). The asymptotic distribution of the numerals in the decimal representation of the squares of the natural numbers. \textit{Math. Z.} {39}(1): 146--156. DOI: 10.1007/BF01201350.

\bibitem{borel}
Borel, E. (1909). Les probabilit{\'e}s d{\'e}nombrables et leurs applications arithm{\'e}\-ti\-ques. \textit{Rend. Circ. Mat. Palermo} {27}(1): 247--271.

\bibitem{bugeaud}
Bugeaud, Y. (2012). \textit{Distribution Modulo One and Diophantine Approximation}. Cambridge Tracts in Mathematics, Vol.~193. Cambridge, UK: Cambridge Univ. Press.

\bibitem{bumby}
Bumby, R. T. (1975). A distribution property for linear recurrence of the second order. \textit{Proc. Amer. Math. Soc.} {50}(1): 101--106. DOI: 10.1090/S0002-9939-1975-0369240-X.

\bibitem{champernowne}
Champernowne, D. G. (1933). The construction of decimals normal in the scale of ten. \textit{J. Lond. Math. Soc.} {8}(4): 254--260. DOI: 10.1112/jlms/s1-8.4.254.

\bibitem{clanin-rayman}
Clanin, J., Rayman, M. (2025). Finite state dimension and the Davenport--Erd\H{o}s theorem. Preprint arXiv:2506.02332~[cs.IT].
DOI: 10.48550/arXiv.2506.02332.

\bibitem{copeland-erdos}
Copeland, A. H., Erd\H{o}s, P. (1946). Note on normal numbers. \textit{Bull. Amer. Math. Soc.} {52}(10): 857--860.

\bibitem{davenport-erdos}
Davenport, H., Erd\H{o}s, P. (1952). Note on normal decimals. \textit{Canad. J. Math.} {4}: 58--63. DOI: 10.4153/CJM-1952-005-3.

\bibitem{diaconis}
Diaconis, P. (1977). The distribution of leading digits and uniform distribution mod $1$. \textit{Ann. Probab.} {5}(1): 72--81. DOI: 10.1214/aop/1176995891.

\bibitem{diaconis-fulman}
Diaconis, P., Fulman, J. (2009). Carries, shuffling, and an amazing matrix. \textit{Amer. Math. Monthly} {116}(9): 788--803. DOI: 10.4169/000298909X474864.

\bibitem{good-1953}
Good, I. J. (1953). The serial test for sampling numbers and other tests for randomness. \textit{Proc. Cambridge Philos. Soc.} {49}(2): 276--284. DOI: 10.1017/S030500410002836X.

\bibitem{jacobson}
Jacobson, E. T. (1992). Distribution of the Fibonacci numbers mod $2^{k}$. \textit{Fibonacci Quart.} {30}(3): 211--215. DOI: 10.1080/00150517.1992.12429344.

\bibitem{kuipers-niederreiter}
Kuipers, L., Niederreiter, H. (1974). \textit{Uniform Distribution of Sequences}. New York, NY: John Wiley \& Sons.

\bibitem{kuipers}
Kuipers, L., Shiue, J. (1972). A distribution property of the sequence of Fibonacci numbers. \textit{Fibonacci Quart.} {10}(4): 375--376.

\bibitem{levin}
Levin, M. (1999). On the discrepancy estimate of normal numbers. \textit{Acta Arith.} {88}(2): 99--111. DOI: 10.4064/aa-88-2-99-111.

\bibitem{normality-criterion}
Madritsch, M. G., Thuswaldner, J. M., Tichy, R. F. (2008). Normality of numbers generated by the values of entire functions. \textit{J. Number Theory} {128}(5): 1127--1145. DOI: 10.1016/j.jnt.2007.04.005.

\bibitem{nakai-shiokawa-1990}
Nakai, I., Shiokawa, I. (1990). A class of normal numbers. \textit{Japan. J. Math. (N.\,S.)} {16}(1): 17--29. DOI: 10.4099/math1924.16.17.

\bibitem{nakai-shiokawa-1992}
Nakai, I., Shiokawa, I. (1992). Discrepancy estimates for a class of normal numbers. \textit{Acta Arith.} {62}(3): 271--284. DOI: 10.4064/aa-62-3-271-284.

\bibitem{niederreiter}
Niederreiter, H. (1972). Distribution of Fibonacci numbers mod $5^{k}$. \textit{Fibonacci Quart.} {10}(4): 373--374.

\bibitem{pollack-vandehey}
Pollack, P., Vandehey, J. (2015). Besicovitch, bisection, and the normality of $0.(1)(4)(9)(16)(25)\dots$. \textit{Amer. Math. Monthly} {122}(8): 757--765. DOI: 10.4169/amer.math.monthly.122.8.757.

\bibitem{phi-sigma-lambda}
Pollack, P., Vandehey, J. (2015). Some normal numbers generated by arithmetic functions. \textit{Canad. Math. Bull.} {58}(1): 160--173. DOI: 10.4153/CMB-2014-047-2.

\bibitem{spilker}
Spilker, J. (2003). Die Ziffern der Fibonacci-Zahlen. \textit{Elem. Math.} {58}(1): 26--33. DOI: 10.1007/s000170300003.

\bibitem{stoneham}
Stoneham, R. G. (1970). A general arithmetic construction of transcendental non-Liouville normal numbers from rational functions. \textit{Acta Arith.} {16}(3): 239--254. DOI: 10.4064/aa-16-3-239-254.

\bibitem{wall}
Wall, D. D. (1960). Fibonacci series modulo $m$. \textit{Amer. Math. Monthly} {67}(6): 525--532. DOI: 10.2307/2309169.

\end{thebibliography}
\end{document}